\documentclass[11pt]{amsart}%
\usepackage{amsmath}
\usepackage{amssymb}
\usepackage{amsfonts}
\usepackage{graphicx}%
\newtheorem{theorem}{Theorem}[section]
\theoremstyle{plain}

\newtheorem{lemma}[theorem]{Lemma}
\newtheorem{proposition}[theorem]{Proposition}

\numberwithin{equation}{section}

\DeclareMathOperator{\diver}{div}
\begin{document}
\title[Supercritical problem]{A supercritical elliptic problem in a cylindrical shell}
\author{M\'{o}nica Clapp}
\address{Instituto de Matem\'{a}ticas, Universidad Nacional Aut\'{o}noma de M\'{e}xico,
Circuito Exterior, C.U., 04510 M\'{e}xico D.F., Mexico.}
\email{monica.clapp@im.unam.mx}
\author{Andrzej Szulkin}
\address{Department of Mathematics, Stockholm University, 106 91 Stockholm, Sweden}
\email{andrzejs@math.su.se}
\thanks{This research is partially supported by CONACYT grant 129847,
UNAM-DGAPA-PAPIIT grant IN106612 (M\'{e}xico), and the Swedish Research Council.}
\date{\today}

\begin{abstract}
We consider the problem%
\[
-\Delta u=\left\vert u\right\vert ^{p-2}u\text{ \ in }\Omega,\quad u=0\text{
\ on }\partial\Omega,
\]
where $\Omega:=\{(y,z)\in\mathbb{R}^{m+1}\times\mathbb{R}^{N-m-1}%
:0<a<\left\vert y\right\vert <b<\infty\}$, $0\leq m\leq N-1$ and $N\geq2.$ Let
$2_{N,m}^{\ast}:=2(N-m)/(N-m-2)$ if $m<N-2$ and $2_{N,m}^{\ast}:=\infty$ if
$m=N-2$ or $N-1.$ We show that $2_{N,m}^{\ast}$ is the true critical exponent
for this problem, and that there exist nontrivial solutions if $2<p<2_{N,m}%
^{\ast}$ but there are no such solutions if $p\geq2_{N,m}^{\ast}$.

\bigskip


\end{abstract}
\maketitle

\begin{center}
\emph{To Bernhard Ruf on his birthday, with our friendship and great esteem.}
\end{center}

\baselineskip15pt

\section{Introduction}

Consider the Lane-Emden-Fowler problem%
\begin{equation}
-\Delta u=\left\vert u\right\vert ^{p-2}u\text{ \ in }\mathcal{D}%
\text{,\qquad}u=0\text{ \ on }\partial\mathcal{D}, \label{model}%
\end{equation}
where $\mathcal{D}$ is a smooth domain in $\mathbb{R}^{N}$ and $p>2$.

If $\mathcal{D}$ is bounded it is well-known that this problem has at least
one positive solution and infinitely many sign changing solutions when $p$ is
smaller than the critical Sobolev exponent $2^{\ast},$ defined as $2^{\ast
}:=\frac{2N}{N-2}$ if $N\geq3$ and as $2^{\ast}:=\infty$ if $N=1$ or 2. In
contrast, the existence of solutions for $p\geq2^{\ast}$ is a delicate issue.
Pohozhaev's identity \cite{po} implies that problem (\ref{model}) has no
nontrivial solution if the domain $\mathcal{D}$ is strictly starshaped. On the
other hand, Bahri and Coron \cite{bc} proved that a positive solution to
(\ref{model}) exists if $p=2^{\ast}$ and $\mathcal{D}$ is bounded and has
nontrivial reduced homology with $\mathbb{Z}/2$ coefficients.

One may ask whether this last statement is also true for $p>2^{\ast}.$
Passaseo showed in \cite{pa1,pa2} that this is not so: for each $1\leq m<N-2$
he exhibited a bounded smooth domain $\mathcal{D}$ which is homotopy
equivalent to the $m$-dimensional sphere, in which problem (\ref{model}) has
infinitely many solutions if $p<2_{N,m}^{\ast}:=\frac{2(N-m)}{N-m-2}$ and does
not have a nontrivial solution if $p\geq2_{N,m}^{\ast}.$ Examples of domains
with richer homology were recently given by Clapp, Faya and Pistoia in
\cite{cfp}. Wei and Yan established in \cite{wy} the existence of infinitely
many positive solutions for $p=2_{N,m}^{\ast}$ in some bounded domains. For
$p$ slightly below $2_{N,m}^{\ast}$ solutions concentrating along an
$m$-dimensional manifold were recently obtained in \cite{acp,dmp}. Note that
$2_{N,m}^{\ast}$ is the critical Sobolev exponent in dimension $N-m.$ It is
called the $(m+1)$-st critical exponent for problem (\ref{model}).

The purpose of this note is to exhibit unbounded domains in which this problem
has the behavior described by Passaseo.

We consider the problem
\begin{equation}
\left\{
\begin{array}
[c]{ll}%
-\Delta u=\left\vert u\right\vert ^{p-2}u & \text{in }\Omega,\\
\quad\ \ u=0 & \text{on }\partial\Omega,\\
\left\vert \nabla u\right\vert ^{2},\left\vert u\right\vert ^{p}\in
L^{1}(\Omega), &
\end{array}
\right.  \label{prob}%
\end{equation}
in a cylindrical shell
\[
\Omega:=\{x=(y,z)\in\mathbb{R}^{m+1}\times\mathbb{R}^{N-m-1}:a<\left\vert
y\right\vert <b\},\qquad0<a<b<\infty,
\]
for $p>2$.

If $m=N-1$ or $N-2$, we set $2_{N,m}^{\ast}:=\infty$. First note that if
$m=N-1$ then $\Omega=\{x\in\mathbb{R}^{N}:a<\left\vert x\right\vert <b\}$, and
a well-known result by Kazdan and Warner \cite{kw} asserts that \eqref{prob}
has infinitely many radial solutions for any $p>2$. In the other extreme case,
where $m=0,$ the domain $\Omega$ is the union of two disjoint strips
$(a,b)\times\mathbb{R}^{N-1}$ and $(-b,-a)\times\mathbb{R}^{N-1}.$ Each of
them is starshaped, so there are no solutions for $p\geq2_{N,0}^{\ast}%
=2^{\ast}.$ \ Esteban showed in \cite{e} that there are infinitely many
solutions in $(a,b)\times\mathbb{R}^{N-1}$ if $N\geq3$ and $p<2^{\ast}$, and
one positive solution if $N=2$ (in fact, she considered a more general
problem). These solutions are axially symmetric, i.e. $u(y,z)=u(y,\left\vert
z\right\vert )$ for all $(y,z)\in\Omega.$

Here we study the remaining cases, i.e., $1\leq m\leq N-2.$ Our first result
states the nonexistence of solutions other than $u=0,$ if $p\geq2_{N,m}^{\ast
}.$

\begin{theorem}
\label{thm1}If $1\leq m<N-2$ and $p\geq2_{N,m}^{\ast}$, then problem
\eqref{prob} does not have any nontrivial solution $u\in\mathcal{C}^{2}%
(\Omega)\cap\mathcal{C}^{1}(\overline{\Omega})$.
\end{theorem}

Our next result shows that solutions $u\neq0$ do exist if $2<p<2_{N,m}^{\ast}$.

As usual, we write $O(k)$ for the group of linear isometries of $\mathbb{R}%
^{k}$ (represented by orthogonal $k\times k$-matrices). Recall that if $G$ is
a closed subgroup of $O(N)$ then a subset $X$ of $\mathbb{R}^{N}$ is
$G$-invariant if $gX=X$ for every $g\in G,$ and a function $u:X\rightarrow
\mathbb{R}$ is called $G$-invariant provided $u(gx)=u(x)$ for all $g\in G,$
$x\in X.$

Note that $\Omega$ is $\left[  O(m+1)\times O(N-m-1)\right]  $-invariant for
the obvious action given by $(g,h)(y,z):=(gy,hz)$ for all $g\in O(m+1),$ $h\in
O(N-m-1),$ $y\in\mathbb{R}^{m+1},$ $z\in\mathbb{R}^{N-m-1}.$

\begin{theorem}
\label{thm2} (i) If $1\leq m<N-2$ and $2<p<2_{N,m}^{\ast}$, then problem
\eqref{prob} has infinitely many $\left[  O(m+1)\times O(N-m-1)\right]
$-invariant solutions and one of these solutions is positive.\newline(ii) If
$1\leq m=N-2$ and $2<p<\infty$, then problem \eqref{prob} has a positive
$\left[  O(N-1)\times O(1)\right]  $-invariant solution.
\end{theorem}

In Section \ref{poid} we prove Theorem \ref{thm1}. Theorem \ref{thm2} is
proved in Section~\ref{pf2}. We conclude the paper with a multiplicity result
and an open question in Section \ref{rem}.

\section{A Poho\v{z}aev identity and the proof of Theorem \ref{thm1}}

\label{poid}We prove Theorem \ref{thm1}\ by adapting Passaseo's argument in
\cite{pa1,pa2}, see also \cite{cfp}. The proof relies on the following special
case of a Poho\v{z}aev type identity due to Pucci and Serrin \cite{ps}.

For $(u,v)\in\mathbb{R}\times\mathbb{R}^{N}$ we set%
\[
\phi(u,v):=\frac{1}{2}\left\vert v\right\vert ^{2}-\frac{1}{p}\left\vert
u\right\vert ^{p}.
\]

\begin{lemma}
If $u\in\mathcal{C}^{2}(\Omega)$ satisfies $-\Delta u=\left\vert u\right\vert
^{p-2}u$ \ in $\Omega$ then, for every $\chi\in\mathcal{C}^{1}(\overline
{\Omega},\mathbb{R}^{N}),$ the equality%
\begin{equation}
\left(  \diver\chi\right)  \phi(u,\nabla u)-D\chi\left[  \nabla u\right]
\cdot\nabla u=\diver\left[  \phi(u,\nabla u)\chi-(\chi\cdot\nabla u)\nabla
u\right]  \label{pointwise}%
\end{equation}
holds true.
\end{lemma}

\begin{proof}
Put $\chi=(\chi_{1},\ldots,\chi_{N})$, denote the partial derivative with
respect to $x_{k}$ by $\partial_{k}$ and let LHS and RHS denote the left- and
the right-hand side of \eqref{pointwise}. Then
\[
\text{LHS}=\left(  \diver\chi\right)  \phi(u,\nabla u)-\sum_{j,k}\partial
_{k}\chi_{j}\,\partial_{j}u\,\partial_{k}u
\]
and
\begin{align*}
\text{RHS}  &  =\left(  \diver\chi\right)  \phi(u,\nabla u)+\sum_{j,k}\chi
_{k}\,\partial_{j}u\,\partial_{jk}^{2}u-|u|^{p-2}u\,\nabla u\cdot\chi\\
&  \qquad-(\nabla u\cdot\chi)\Delta u-\sum_{j,k}\partial_{k}\chi_{j}%
\,\partial_{j}u\,\partial_{k}u-\sum_{j,k}\chi_{j}\,\partial_{k}u\,\partial
_{jk}^{2}u\\
&  =\left(  \diver\chi\right)  \phi(u,\nabla u)-(\nabla u\cdot\chi)(\Delta
u+|u|^{p-2}u)-\sum_{j,k}\partial_{k}\chi_{j}\,\partial_{j}u\,\partial_{k}u.
\end{align*}
Since $-\Delta u=\left\vert u\right\vert ^{p-2}u$, the conclusion follows.
\end{proof}

Using a well-known truncation argument, we can now prove the following result.

\begin{proposition}
\label{proppohozhaev}Assume that $\chi\in\mathcal{C}^{1}(\overline{\Omega
},\mathbb{R}^{N})$ has the following properties:

\begin{enumerate}
\item[(a)] $\chi\cdot\nu$ is bounded on $\partial\Omega$, where $\nu(s)$ is
the outer unit normal at $s\in\partial\Omega,$

\item[(b)] $\left\vert \chi(x)\right\vert \leq\left\vert x\right\vert $ for
every $x\in\Omega,$

\item[(c)] $\diver\chi$ is bounded in $\Omega,$

\item[(d)] $\left\vert D\chi(x)\xi\cdot\xi\right\vert \leq\left\vert
\xi\right\vert ^{2}$ for all $x\in\Omega,$ $\xi\in\mathbb{R}^{N}$.
\end{enumerate}

\noindent Then every solution $u\in\mathcal{C}^{2}(\Omega)\cap\mathcal{C}%
^{1}(\overline{\Omega})$ of \emph{(\ref{prob})} satisfies%
\begin{equation}
\frac{1}{2}\int_{\partial\Omega}\left\vert \nabla u\right\vert ^{2}\chi
\cdot\nu=-\int_{\Omega}\left(  \diver\chi\right)  \phi(u,\nabla u)+\int
_{\Omega}D\chi\left[  \nabla u\right]  \cdot\nabla u. \label{PS}%
\end{equation}

\end{proposition}

\begin{proof}
Choose $\psi\in\mathcal{C}^{\infty}(\mathbb{R})$ such that $0\leq\psi
(t)\leq1,$ $\psi(t)=1$ if $\left\vert t\right\vert \leq1$ and $\psi(t)=0$ if
$\left\vert t\right\vert \geq2.$ For each $k\in\mathbb{N}$ define%
\[
\psi_{k}(x):=\psi\left(  \frac{\left\vert x\right\vert ^{2}}{k^{2}}\right)
\text{\qquad and\qquad}\chi^{k}(x):=\psi_{k}(x)\chi(x).
\]
Note that there is a constant $c_{0}>0$ such that%
\begin{equation}
\left\vert x\right\vert \left\vert \nabla\psi_{k}(x)\right\vert \leq
c_{0}\text{\qquad for all }x\in\mathbb{R}^{N},\text{ }k\in\mathbb{N}.
\label{psik}%
\end{equation}
Next, choose a sequence of bounded smooth domains $\Omega_{k}\subset\Omega$
such that%
\begin{equation}
\label{omegak}\Omega_{k}\supset\Omega\cap\overline{B_{2k}(0)}.
\end{equation}
Integrating (\ref{pointwise}) with $\chi:=\chi^{k}$ in $\Omega_{k}$ and using
the divergence theorem and Lemma \ref{pointwise} we obtain%
\begin{align*}
&  \int_{\Omega_{k}}\left(  \diver\chi^{k}\right)  \phi(u,\nabla
u)-\int_{\Omega_{k}}D\chi^{k}\left[  \nabla u\right]  \cdot\nabla u=\\
&  \int_{\partial\Omega_{k}}\left[  \phi(u,\nabla u)\left(  \chi^{k}\cdot
\nu^{k}\right)  -(\chi^{k}\cdot\nabla u)\left(  \nabla u\cdot\nu^{k}\right)
\right]  ,
\end{align*}
where $\nu^{k}$ is the outer unit normal to $\Omega_{k}.$ Property
(\ref{omegak}) implies that $\chi^{k}=0$ in $\overline{\Omega\smallsetminus
\Omega_{k}},$ so we may replace $\Omega_{k}$ by $\Omega$, $\partial\Omega_{k}$
by $\partial\Omega$ and $\nu^{k}$ by $\nu$ in the previous identity. Moreover,
since $u=0$ on $\partial\Omega$, we have that%
\[
\nabla u=\left(  \nabla u\cdot\nu\right)  \nu\text{ \ on }\partial\Omega.
\]
Therefore,%
\begin{align}
&  \int_{\Omega}\left(  \diver\chi^{k}\right)  \phi(u,\nabla u)-\int_{\Omega
}D\chi^{k}\left[  \nabla u\right]  \cdot\nabla u=\nonumber\\
&  \int_{\partial\Omega}\left[  \phi(u,\nabla u)\left(  \chi^{k}\cdot
\nu\right)  -(\chi^{k}\cdot\nabla u)\left(  \nabla u\cdot\nu\right)  \right]
=\label{PSbdd}\\
&  \int_{\partial\Omega}\left[  \phi(u,\nabla u)-\left\vert \nabla
u\right\vert ^{2}\right]  \left(  \chi^{k}\cdot\nu\right)  =-\frac{1}{2}%
\int_{\partial\Omega}\left\vert \nabla u\right\vert ^{2}\psi_{k}(x)\left(
\chi\cdot\nu\right)  .\nonumber
\end{align}
Since $\diver\chi^{k}=\psi_{k}\diver\chi+\nabla\psi_{k}\cdot\chi$, using
(\ref{psik}) and properties (b) and (c) we obtain%
\begin{equation}
\left\vert \diver\chi^{k}\right\vert \leq\left\vert \diver\chi\right\vert
+\left\vert \nabla\psi_{k}\right\vert \left\vert \chi\right\vert
\leq\left\vert \diver\chi\right\vert +c_{0}\leq c_{1}\text{ \ in }\Omega.
\label{div}%
\end{equation}
Similarly, since
\[
D\chi^{k}(x)\xi\cdot\xi=\psi_{k}(x)D\chi(x)\xi\cdot\xi+\left(  \nabla\psi
_{k}\cdot\xi\right)  \left(  \chi\cdot\xi\right)  ,
\]
property (d) yields%
\begin{equation}
\left\vert D\chi^{k}(x)\xi\cdot\xi\right\vert \leq\left(  1+c_{0}\right)
\left\vert \xi\right\vert ^{2}\text{\qquad for all }x\in\Omega,\ \xi
\in\mathbb{R}^{N}. \label{D}%
\end{equation}
Inequalities (\ref{div}), (\ref{D}) and property (a) allow us to apply
Lebesgue's dominated convergence theorem to the left- and the right-hand side
of (\ref{PSbdd}) to obtain%
\[
\int_{\Omega}\left(  \diver\chi\right)  \phi(u,\nabla u)-\int_{\Omega}%
D\chi\left[  \nabla u\right]  \cdot\nabla u=-\frac{1}{2}\int_{\partial\Omega
}\left\vert \nabla u\right\vert ^{2}\left(  \chi\cdot\nu\right)  ,
\]
as claimed.
\end{proof}

\begin{proof}
[Proof of Theorem \ref{thm1}]Let $\varphi(t)=\frac{1}{m+1}\left[  1-(\frac
{a}{t})^{m+1}\right]  $ be the solution to the boundary value problem%
\[
\left\{
\begin{array}
[c]{ll}%
\varphi^{\prime}(t)t+(m+1)\varphi(t)=1, & t\in(0,\infty),\\
\varphi(a)=0. &
\end{array}
\right.
\]
Define%
\begin{equation}
\chi(y,z):=(\varphi(\left\vert y\right\vert )y,z).\label{vf}%
\end{equation}
Then, if $\nu$ denotes the outer unit normal on $\partial\Omega$,%
\begin{equation}
\left(  \chi\cdot\nu\right)  (y,z)=\left\{
\begin{array}
[c]{ll}%
0 & \text{if }\left\vert y\right\vert =a,\\
\frac{1}{m+1}\left[  1-(\frac{a}{b})^{m+1}\right]  b & \text{if }\left\vert
y\right\vert =b.
\end{array}
\right.  \label{normal}%
\end{equation}
So property (a) of Proposition \ref{proppohozhaev}\ holds. Clearly, (b) holds.
Now,%
\begin{equation}
\diver\chi(y,z)=\left[  \varphi^{\prime}(\left\vert y\right\vert )\left\vert
y\right\vert +(m+1)\varphi(\left\vert y\right\vert )\right]
+N-m-1=N-m.\label{divN-m}%
\end{equation}
In particular, (c) holds. To prove (d) notice that $\chi$ is $O(m+1)$%
-equivariant, i.e.
\[
\chi(gy,z)=g\chi(y,z)\text{\qquad for every }g\in O(m+1).
\]
Therefore, $g\circ D\chi(y,z)=D\chi(gy,z)\circ g$ and, hence,%
\[
\left\langle D\chi\left(  y,z\right)  \left[  \xi\right]  ,\xi\right\rangle
=\left\langle g\left(  D\chi\left(  y,z\right)  \left[  \xi\right]  \right)
,g\xi\right\rangle =\left\langle D\chi\left(  gy,z\right)  [g\xi
],g\xi\right\rangle
\]
for all $\xi\in\mathbb{R}^{N}.$ Thus, it suffices to show that the inequality
(d) holds for $y=(t,0,\ldots,0)$ with $t\in(a,b).$ A straightforward
computation shows that, for such $y,$ $D\chi(y)$ is a diagonal matrix whose
diagonal entries are $a_{11}=1-m\varphi(t)$, $a_{jj}=\varphi(t)$ for
$j=2,\ldots,m+1,$ and $a_{jj}=1$ for $j=m+2,\ldots,N.$ Since $a_{jj}\in(0,1],$%
\begin{equation}
0<\left\langle D\chi\left(  y,z\right)  \left[  \xi\right]  ,\xi\right\rangle
\leq\left\vert \xi\right\vert ^{2}\text{\quad for all }\xi\in\mathbb{R}%
^{N}\smallsetminus\{0\}\label{quad}%
\end{equation}
and (d) follows. From (\ref{normal}), (\ref{PS}), (\ref{quad}) and
(\ref{divN-m}) we obtain%
\begin{align*}
0 &  <\frac{1}{2}\int_{\partial\Omega}\left\vert \nabla u\right\vert ^{2}%
\chi\cdot\nu=-\int_{\Omega}\left(  \diver\chi\right)  \phi(u,\nabla
u)+\int_{\Omega}D\chi\left[  \nabla u\right]  \cdot\nabla u\\
&  \leq(N-m)\int_{\Omega}\left[  \frac{1}{p}\left\vert u\right\vert ^{p}%
-\frac{1}{2}\left\vert \nabla u\right\vert ^{2}\right]  +\int_{\Omega
}\left\vert \nabla u\right\vert ^{2}\\
&  =(N-m)\left(  \frac{1}{p}-\frac{1}{2}+\frac{1}{N-m}\right)  \int_{\Omega
}\left\vert \nabla u\right\vert ^{2}.
\end{align*}
The first (strict) inequality follows from the unique continuation property
\cite{jk,gl}. This immediately implies that $p<2_{N,m}^{\ast}.$ 
\end{proof}

\section{The proof of Theorem \ref{thm2}}

\label{pf2}An $O(m+1)$-invariant function $u(y,z)=v(\left\vert y\right\vert
,z)$ solves problem (\ref{prob}) if and only if $v=v(r,z)$ solves
\begin{equation}
\left\{
\begin{array}
[c]{cl}%
-\Delta v-\frac{m}{r}\frac{\partial v}{\partial r}=|v|^{p-2}v & \text{in}%
\ (a,b)\times\mathbb{R}^{N-m-1}=:\mathcal{S},\\
v=0 & \text{on}\ \{a,b\}\times\mathbb{R}^{N-m-1}=\partial\mathcal{S},
\end{array}
\right.  \label{eq1}%
\end{equation}
and $|\nabla v|^{2},|v|^{p}\in L^{1}(\mathcal{S})$. Problem \eqref{eq1} can be
rewritten as%
\begin{equation}
-\diver(r^{m}\nabla v)=r^{m}|v|^{p-2}v\quad\text{in}\ \mathcal{S},\qquad
v=0\quad\text{on}\ \partial\mathcal{S}. \label{eqdiv}%
\end{equation}

By Poincar\'{e}'s inequality (see Lemma 3 in \cite{e}) and since $a<r<b$, the
norms%
\begin{equation}
\left\Vert v\right\Vert _{m}:=\left(  \int_{\mathcal{S}}r^{m}\left\vert \nabla
v\right\vert ^{2}\right)  ^{1/2}\text{\qquad and\qquad}\left\vert v\right\vert
_{m,p}:=\left(  \int_{\mathcal{S}}r^{m}\left\vert v\right\vert ^{p}\right)
^{1/p} \label{newnorms}%
\end{equation}
are equivalent to those of $H_{0}^{1}(\mathcal{S})$ and $L^{p}(\mathcal{S})$ respectively.

Consider the functional $I(v):=\Vert v\Vert_{m}^{2}$ restricted to
\[
M:=\{v\in H_{0}^{1}(\mathcal{S}):\left\vert v\right\vert _{m,p}=1\}.
\]
Then $M$ is a $C^{2}$-manifold, and $v$ is a critical point of $I|_{M}$ if and
only if $v\in H_{0}^{1}(\mathcal{S})$ and $\Vert v\Vert_{m}^{2/(p-2)}v$ is a
nontrivial solution to \eqref{eqdiv}. Note that $I|_{M}$ is bounded below by a
positive constant.

\begin{proof}
[Proof of Theorem \ref{thm2} (i)]Assume that $1\leq m<N-2$ and $2<p<2_{N,m}%
^{\ast}.$ Set $G:=O(N-m-1)$ and denote by $H_{0}^{1}(\mathcal{S})^{G}$ and
$L^{p}(\mathcal{S})^{G}$ the subspaces of $H_{0}^{1}(\mathcal{S})$ and
$L^{p}(\mathcal{S})$ respectively, consisting of functions $v$ such that
$v(r,gz)=v(r,z)$ for all $g\in G$. Esteban and Lions showed in \cite{el} that,
for these values of $m$ and $p,$ $H_{0}^{1}(\mathcal{S})^{G}$ is compactly
embedded in $L^{p}(\mathcal{S})^{G}$ (see also Theorem 1.24 in \cite{wi}). So
$H_{0}^{1}(\mathcal{S})^{G}$ is compactly embedded in $L^{p}(\mathcal{S})^{G}$
for the norms (\ref{newnorms}) as well.

Let%
\[
M^{G}:=\{v\in H_{0}^{1}(\mathcal{S})^{G}:\left\vert v\right\vert _{m,p}=1\}.
\]
It follows from the principle of symmetric criticality \cite[Theorem 1.28]{wi}
that the critical points of $I|_{M^{G}}$ are also critical points of $I|_{M}.$
The manifold $M^{G}$ is radially diffeomorphic to the unit sphere in
$H_{0}^{1}(\mathcal{S})^{G},$ so its Krasnoselskii genus is infinite. A
standard argument, using the compactness of the embedding $H_{0}%
^{1}(\mathcal{S})^{G}\hookrightarrow L^{p}(\mathcal{S})^{G}$ for the norms
(\ref{newnorms}), shows that $I|_{M^{G}}$ satisfies the Palais-Smale
condition. Hence $I|_{M^{G}}$ has infinitely many critical points (see e.g.
Theorem II.5.7 in \cite{str}). It can also be shown by a well-known argument
that the critical values of $I|_{M^{G}}$ tend to infinity (see e.g.
Proposition 9.33 in~\cite{ra}).

It remains to show that \eqref{eqdiv} has a positive solution. The argument is
again standard: since $I|_{M^{G}}$ satisfies the Palais-Smale condition,
\[
c_{0}^{G}:=\inf\{I(v):v\in M^{G}\}
\]
is attained at some $v_{0}$. Since $I(v)=I(\left\vert v\right\vert )$ and
$\left\vert v\right\vert \in M^{G}$ if $v\in M^{G}$, we have that
$I(|v_{0}|)=c_{0}^{G}$ and we may assume $v_{0}\geq0$. The maximum principle
applied to the corresponding solution $u_{0}$ of \eqref{prob} implies
$u_{0}>0$.
\end{proof}

If $m=N-2$, then $G=O(1)$ and it is easy to see that the space $H_{0}%
^{1}(\mathcal{S})^{G}$ is not compactly embedded in $L^{p}(\mathcal{S})^{G}$.
So part (ii) of Theorem \ref{thm2} requires a different argument.

\begin{proof}
[Proof of Theorem \ref{thm2} (ii)]Assume that $1\leq m=N-2$ and $2<p<\infty.$
We shall show that
\[
c_{0}:=\inf\{I(v):v\in M\}
\]
is attained. Clearly, a minimizing sequence $(v_{n})$ is bounded, so we may
assume that $v_{n}\rightharpoonup v$ weakly in $H_{0}^{1}(\mathcal{S})$.
According to P.-L. Lions' lemma \cite[Lemma 1.21]{wi} either $v_{n}%
\rightarrow0$ strongly in $L^{p}(\mathcal{S}),$ which is impossible because
$v_{n}\in M$, or there exist $\delta>0$ and $(r_{n},z_{n})\in\lbrack
a,b]\times\mathbb{R}$ such that, after passing to a subsequence if necessary,%
\begin{equation}
\int_{B_{1}(r_{n},z_{n})}v_{n}^{2}\geq\delta.\label{delta}%
\end{equation}
Here $B_{1}(r_{n},z_{n})$ denotes the ball of radius 1 and center at
$(r_{n},z_{n})$. Since the problem is invariant with respect to translations
along the $z$-axis, replacing $v_{n}(r,z)$ by $v_{n}(r,z+z_{n})$, we may
assume the center of the ball above is $(r_{n},0)$. It follows that for this -
translated - sequence the weak limit $v$ cannot be zero due to \eqref{delta}
and the compactness of the embedding of $H_{0}^{1}(\mathcal{S})$ in
$L_{loc}^{2}(\mathcal{S})$. Passing to a subsequence once more, we have that
$v_{n}(x)\rightarrow v(x)$ a.e. It follows from the Brezis-Lieb lemma
\cite[Lemma 1.32]{wi} that
\[
1=|v_{n}|_{m,p}^{p}=\lim_{n\rightarrow\infty}|v_{n}-v|_{m,p}^{p}+|v|_{m,p}%
^{p}.
\]
Using this identity and the definition of $c_{0}$ we obtain%
\begin{align*}
c_{0} &  =\lim_{n\rightarrow\infty}\Vert v_{n}\Vert_{m}^{2}=\lim
_{n\rightarrow\infty}\Vert v_{n}-v\Vert_{m}^{2}+\Vert v\Vert_{m}^{2}\geq
c_{0}\left(  \lim_{n\rightarrow\infty}|v_{n}-v|_{m,p}^{2}+|v|_{m,p}%
^{2}\right)  \\
&  =c_{0}\left(  (1-|v|_{m,p}^{p})^{2/p}+(|v|_{m,p}^{p})^{2/p}\right)  \geq
c_{0}(1-|v|_{m,p}^{p}+|v|_{m,p}^{p})^{2/p}=c_{0}.
\end{align*}
Since $v\neq0$, it follows that $|v_{n}-v|_{m,p}\rightarrow0$ and
$|v|_{m,p}=1$. So $v\in M$ and, as $c_{0}=\lim_{n\rightarrow\infty}%
I(v_{n})\geq I(v)$, we must have $I(v)=c_{0}$.

So the infimum is attained at $v$ and using the moving plane method
\cite[Appendix C]{wi}, we may assume, after translation, that $v(r,-z)=v(r,z)$%
, i.e. $v\in H_{0}^{1}(\mathcal{S})^{O(1)}$. As in the preceding proof,
replacing $v$ by $|v|$, we obtain a positive solution.
\end{proof}

\section{Further solutions and an open question}

\label{rem}

If $1\leq m=N-2$ and $p\in(2,2_{N,m}^{\ast})$, the method we have used to
prove Theorem \ref{thm2}\ only guarantees the existence of two solutions to
problem \eqref{prob}, one positive and one negative, up to translations along
the $z$-axis. However, if $p\in(2,2^{\ast})$, then it is possible to show that
there are infinitely many solutions, which are not radial in $y,$ but have
other prescribed symmetry properties.

Write $y=(y^{1},y^{2})\in\mathbb{R}^{2}\times\mathbb{R}^{m-1}\equiv
\mathbb{R}^{m+1}$ and identify $\mathbb{R}^{2}$ with the complex plane
$\mathbb{C}$. Following \cite{sw}, we denote by $G_{k}$, $k\geq3$, the
subgroup of $O(2)$ generated by two elements $\alpha,\beta$ which act on
$\mathbb{C}$ by
\[
\alpha y^{1}:=e^{2\pi i/k}y^{1},\text{\qquad}\beta y^{1}:=e^{2\pi
i/k}\overline{y^{1}},
\]
i.e. $\alpha$ is the rotation in $\mathbb{C}$ by the angle $2\pi/k$ and
$\beta$ is the reflection in the line $y_{2}^{1}=\tan(\pi/k)y_{1}^{1}$, where
$y^{1}=y_{1}^{1}+\mathrm{i}y_{2}^{1}\in\mathbb{C}$. Observe that $\alpha
,\beta$ satisfy the relations $\alpha^{k}=\beta^{2}=e$, $\alpha\beta
\alpha=\alpha$. Let $G_{k}$ act on $\mathbb{R}^{N}$ by $gx=(gy^{1},y^{2},z).$

\begin{theorem}
If $1\leq m\leq N-2$ and $2<p<2^{\ast}$ then, for each $k\geq3,$ problem
\eqref{prob} has a solution $u_{k}$ which satisfies%
\begin{equation}
u_{k}(x)=\text{\emph{det}}(g)u_{k}(g^{-1}x)\text{\qquad for all }g\in G_{k},
\label{symm}%
\end{equation}
and $u_{k}\neq u_{j}$ if $k\neq j.$
\end{theorem}

\begin{proof}
Since the approach is taken from \cite{sw}, we give only a brief sketch of the
proof here and refer to Section 2 of \cite{sw} for more details.

The group $G_{k}$ acts on $H_{0}^{1}(\Omega)$ by
\[
(gu)(x):=\text{det}(g)u(g^{-1}x),
\]
where $\text{det}(g)$ is the determinant of $g$. Let%
\[
H_{0}^{1}(\Omega)^{G_{k}}:=\{u\in H_{0}^{1}(\Omega):u(gx)=\text{det}%
(g)u(g^{-1}x)\text{ for all }g\in G_{k}\}
\]
be the fixed point space of this action, and define $I(u):=\int_{\Omega
}\left\vert \nabla u\right\vert ^{2}$ and%
\[
M^{G_{k}}:=\{u\in H_{0}^{1}(\Omega)^{G_{k}}: |u|_{p}=1\}.
\]
By the principle of symmetric criticality the critical points of
$I|_{M^{G_{k}}}$ are nontrivial solutions to problem \eqref{prob} which
satisfy (\ref{symm}). Now we can see as in the proof of part (ii) of Theorem
\ref{thm2} that there exists a minimizer $u_{k}$ for $I$ on the manifold
$M^{G_{k}}$. Moreover, we may assume that $u_{k}$ has exactly $2k$ nodal
domains, see Corollary 2.7 in \cite{sw}. So in particular, $u_{k}\neq u_{j}$
if $k\neq j$.
\end{proof}

The question whether problem \eqref{prob} has infinitely many solutions when
$1\leq m=N-2$ and $p\in\lbrack2^{\ast},2_{N,m}^{\ast})$ remains open. We
believe that the answer is yes, but the proof would require different methods.

\end{document}